\documentclass[reqno]{amsart}
\usepackage{amsfonts}

\newtheorem{theorem}{Theorem}
\theoremstyle{plain}

\newtheorem{definition}{Definition}

\numberwithin{equation}{section}

\begin{document}
\title[New properties of grand amalgam spaces]{New properties of grand amalgam spaces}
\author{A.Turan G\"{u}rkanl\i }
\address{Istanbul Arel University Faculty of Science and Letters Department
of Mathematics and Computer Sciences}
\email{turangurkanli@arel.edu.tr}
\date{}

\begin{abstract}
In $ \left[14\right]$, a new family called grand amalgam space $W( L^{p),\theta},L^{q),\theta })$ of amalgam spaces was defined and investigated properties of these spaces. The present paper is a sequel to my work $[14].$ In this paper, notations are included in Section 1. In Section 2, we introduce another equivalent but discrete definition of grand amalgam space and study properties of these spaces. In Section 3, we determine necessary and sufficient conditions on a locally compact Abelian group $G$ for the grand amalgam spaces $W(L^{p,)\theta},L^ {q,\theta})$ to be an algebra under convolution. 
\end{abstract}

\maketitle

\section{Notations}

Let $\Omega $ be a bounded subset of  $\ R$. The grand Lebesgue space $L^{p)}\left( \Omega \right) $ was
 introduced by Iwaniec-Sbordone in $\left[ 16\right]. $ These authors, in their studies related with the integrability  properties of the Jacobian in a bounded open set  $\Omega $, defined the grand lebesgue space. This Banach space is defined by the norm
\begin{align*}
\left\Vert {f}\right\Vert _{p)}=\sup_{0<\varepsilon \leq p-1}\left(
\varepsilon \int_{\Omega }\left\vert f\right\vert ^{p-\varepsilon }d\mu
\right) ^{\frac{1}{p-\varepsilon }}\tag{1}
\end{align*}
where $1<p<\infty . $ For $0<\varepsilon \leq p-1,$ $L^{p}\left( \Omega \right) \subset L^{p)}\left( \Omega \right) \subset \L^{p-\varepsilon }\left( \Omega \right)$
 hold. For some properties and applications of $L^{p)}\left( \Omega \right) $ spaces we refer to papers $\left[
1\right] ,\left[ 3\right] ,\left[ 5\right] ,\left[ 6\right] ,$ $\left[ 12
\right] $and $\left[ 13\right] .$ We have for all $1<p<\infty $ and $\varepsilon >0$ 
\begin{equation*}
L^{p}\left( \Omega \right) \subset L^{p),\theta }\left( \Omega \right) \subset
L^{p-\varepsilon }\left( \Omega \right) .\tag{2}
\end{equation*}%
 A generalization of the grand Lebesgue
spaces are the spaces $L^{p),\theta }\left( \Omega \right) ,$ $\theta \geq 0,$
defined by the norm
\begin{equation*}
\left\Vert f\right\Vert _{p),\theta ,\Omega }=\left\Vert f\right\Vert
_{p),\theta }=\sup_{0<\varepsilon \leq p-1}\varepsilon ^{\frac{\theta }{
p-\varepsilon }}\left( \int_{\Omega }\left\vert f\right\vert ^{p-\varepsilon
}d\mu \right) ^{\frac{1}{p-\varepsilon }}=\sup_{0<\varepsilon \leq
p-1}\varepsilon ^{\frac{\theta }{p-\varepsilon }}\left\Vert {f}\right\Vert
_{p-\varepsilon };\tag{3}
\end{equation*}%
when $\theta =0$ the space $L^{p),0}\left( \Omega \right) $ reduces to Lebesgue
space $L^{p}\left( \Omega \right) $ and when $\theta =1$ the space $%
L^{p),1}\left( \Omega \right) $ reduces to grand Lebesgue space $L^{p)}\left(
\Omega \right)$,  $\left( \text{see }\left[ 1\right] ,\left[ 12\right]
\right) $ . It is known that the subspace $\overline{C_{0}^{\infty }}$ is not dense in $L^{p)}\left( \Omega \right). $  Its closure consists of functions $f \in L^{p)}\left(\Omega \right)$  such that
\begin{equation}
\lim_{\varepsilon \rightarrow 0}\varepsilon ^{\frac{\theta }{p-\varepsilon }
}\left\Vert f\right\Vert _{p-\varepsilon}=0, \left[12\right]. \tag{4}
\end{equation}
 It is also known that the grand Lebesgue space  $L^{p),\theta }\left( \Omega \right) ,$ is not reflexive.
Different properties and applications of these spaces were discussed in $%
\left[ 1\right] ,\left[ 12\right] $ and $\left[ 13\right] .$

Let\ $1\leq p<\infty ,\theta \geq 0,$ and $J$ be the one of the set of $%
\mathbb{N}
,$ $%
\mathbb{Z}
$ or $%
\mathbb{Z}
^{n}.$ We define the grand Lebesgue sequence space $\ell ^{p),\theta }=\ell
^{p),\theta }\left( J\right) $ by the norm%
\begin{equation*}
\left\Vert u\right\Vert _{\ell ^{p),\theta }\left( J\right)
}=\sup_{0<\varepsilon \leq p-1}\left( \varepsilon ^{\theta
}\sum\limits_{k\in J}\left\vert u_{k}\right\vert ^{p-\varepsilon }\right) ^{
\frac{1}{p-\varepsilon }}=\sup_{0<\varepsilon \leq p-1}\varepsilon ^{\frac{
\theta }{p-\varepsilon }}\left\Vert u\right\Vert _{\ell ^{p-\varepsilon
}\left( J\right) }.\tag{5}
\end{equation*}

Let $p^{^{\prime }}=\frac{p}{p-1},$ $1<p<\infty .$ First consider an
auxiliary space namely $L^{(p^{\prime },\theta }\left( \Omega \right)
,\theta >0,$ defined by 
\begin{equation*}
\left\Vert g\right\Vert _{(p^{^{\prime }},\theta
}=\inf_{g=\sum\limits_{k=1}^{\infty }g_{k}}\left\{
\sum\limits_{k=1}^{\infty }\inf_{0<\varepsilon \leq p-1}\varepsilon ^{-%
\frac{\theta }{p-\varepsilon }}\left( \int_{\Omega }\left\vert
g_{k}\right\vert ^{\left( p-\varepsilon \right) ^{^{\prime }}}dx\right) ^{%
\frac{1}{\left( p-\varepsilon \right) ^{^{\prime }}}}\right\}\tag{6}
\end{equation*}%
where the functions $g_{k},k\in 
\mathbb{N}
,$ being in $\mathcal{M}_{0},$ the set of all real valued measurable
functions, finite a.e. in $\Omega .$ After this definition the generalized
small Lebesgue spaces have been defined by 
\begin{equation*}
L^{p)^{^{\prime }},\theta }\left( \Omega \right) =\left\{ g\in \mathcal{M}%
_{0}:\left\Vert g\right\Vert _{p)^{^{\prime }},\theta }<+\infty \right\} ,
\end{equation*}%
where%
\begin{equation*}
\left\Vert g\right\Vert _{p)^{^{\prime },\theta }}=\sup_{\substack{ 0\leq
\psi \leq \left\vert g\right\vert  \\ \psi \in L^{(p^{\prime },\theta }}}%
\left\Vert \psi \right\Vert _{(p^{^{\prime }},\theta }.\tag{7}
\end{equation*}%
For $\theta =0$ it is $\left\Vert f\right\Vert _{(p^{^{\prime
}},0}=\left\Vert f\right\Vert _{p)^{^{\prime }},\theta },\left[ 4\right] ,
\left[ 11\right] .$

Let $1\leq p,q\leq \infty .$ The space $( L^{p),\theta}) _{loc}$consists
of (classes of) measurable functions $f$ $:\Omega \rightarrow \mathbb{C}$
such that $f\chi _{K}\in L^{p),\theta},$ for any compact subset $K\subset \Omega ,$
 where $\chi _{K}$ is the characteristic function of $f.$ It is a topological
vector space with the family of seminorms $f\rightarrow \left\Vert
f\right\Vert _{p),\theta}.$ Since $L^{p}\subset L^{p),\theta},$ it is easy to show that $
\left( L^{p}\right) _{loc}\hookrightarrow \left( L^{p)}\right) _{loc}.$ 

The grand amalgam space was defined and studied some properties in [14]. This space is defined  as follows:

Let $\Omega $ be a finite subset of real numbers $R$. Also assume that $
1<p,q<\infty $ and $Q\subset \Omega $ is a fix compact subset with nonemty
interior. The grand Wiener amalgam space 
$W( L^{p),\theta},L^{q),\theta })$ consists of all functions (classes of) $
f\in \left( L^{p),\theta }\right) _{loc}$ such that the control function 
\begin{equation*}
F_{f}^{p),\theta }\left( x\right) =F_{f}^{p),\theta ,Q}\left(
x\right) =\left\Vert f\chi _{Q+x}\right\Vert _{p),\theta
}=\sup_{0<\varepsilon \leq p-1}\varepsilon ^{\frac{\theta }{
p-\varepsilon }}\left\Vert f\chi _{Q+x}\right\Vert _{p-\varepsilon }
\end{equation*}
lies in $L^{q),\theta}$, where $x\in R$. The norm of $W( L^{p),\theta
},L^{q),\theta }) $ defined by 
\begin{equation*}
\| f\| _{W( L^{p),\theta},L^{q),\theta}) }=\| F_{f}^{p),\theta}\| _{q),\theta
}=\| \| f\chi _{Q+x}\| _{p),\theta
}\| _{q),\theta}.\tag{8}
\end{equation*}

Since generalized grand Lebesgue spaces are not translation invariant, this
result reflects to the grand Wiener amalgam spaces. Then the definition
 of $W\left( L^{p),\theta},L^{q)}\right)\left(\Omega\right) $ depends on the
choice of $Q.$ 
\par Given any neighbourhood $U$ of $0$ \ in the set of real numbers $ R$, a family $X=(x_i)_{i\in I}\subset R $
 is called $U$-dense if the family $(x_{i}+U)_{i\in I}$ cover $R$. That is
$$\cup_{i\in I}{(x_{i}+U)}=R.$$
The family $X$ is called separated if the sets  $(x_{i}+U)_{i\in I}$ are pairwise disjoint. The family $X$ is called
relatively separated proved that it is finite union of separated sets. We will call a family 
$X=(x_i)_{i\in I}\subset R $ is well-spread in $R$ if it is both $U$-dense and relatively separated.
\par Let $\Omega $ be a finite subset of real numbers $R$. Another equivalent
 but discrete definition of  $W\left( L^{p),\theta},L^{q)}\right)\left(\Omega\right) $ 
is given by using the bounded uniform partition of unity (for short BUPU),
that is a sequence of non-negative functions $\Psi =\left( \psi _{i}\right)
_{i\in I}$ on $\Omega $ corresponding to a sequence $\left( y_{i}\right) $
in $\Omega $ such that
\begin{itemize}
\item[ (a)]  $\sum \psi _{i}\equiv 1,$

\item[ (b)] $\sup_{i\in I}\left\Vert \psi _{i}\right\Vert _{L^{\infty
}}<M,$ for some $M>0,$

\item[ (c)] There exists a compact subset $U\subset \Omega $ with
non-emty interior and $y_{i}\in \Omega $ such that $\sup \left( \psi
_{i}\right) \subset U+y_{i},$

\item[ (d)] For each compact subset $K\subset \Omega,$
$$
\sup_{x\in X}\natural \left\{ i:x\in K+y_{i}\right\} =\sup \natural \left\{
j\in I:K+y_{i}\cap K+y_{j}\neq \phi \right\} <\infty .
$$
\end{itemize}
\section{Discrete Grand Amalgam Space}
\par In this section we will introduce another equivalent but discrete definition of grand amalgam spaces and study properties of these spaces. 
\begin{definition}
Let $\left( \ x _{i}\right) _{i\in I}\subset \Omega$ be well-spread family in $\Omega.$ For any grand Lebesgue space  
 $L^{p),\theta }\left( \Omega \right) ,$ $\theta \geq 0$, we define the associate
 discrete space $\left ( L^{p),\theta }\right) $ as
 $$
\left ( L^{p),\theta }\right)_{d}=\left\{\Lambda :\Lambda =\left(\lambda_{i}\right)_{i \in I} 
\text{ with }
 \sum\limits_{i\in I}\left\vert\lambda\right\vert\chi_{x_{i}+U}\in L^{p),\theta }\left( \Omega \right)\right\}
$$   
with the norm 
$$
\left\Vert\Lambda\right\Vert _{\left ( L^{p),\theta }\right)_{d}}=
\left\Vert\sum\limits_{i\in I}\left\vert\lambda\right\vert\chi_{x_{i}+U}\right\Vert_{p),\theta}
$$

Since $L^{p),\theta}$ is not translation invariant then   $\left ( L^{p),\theta }\right)_{d}$ 
depends on the choice of $U$.
\end{definition}
\begin{theorem}
The discrete of the grand Lebesgue space $ L^{p),\theta }\left(\Omega\right) $ is the the grand Lebesgue sequence space $\ell^{p),\theta}\left(I\right).$
\end{theorem}
\begin{proof}
$\Lambda=\left(\lambda_i\right)_{i\in I}\in \left ( L^{p),\theta }\right)_{d} .$ Then 
\begin{align*}
\sup_{0<\varepsilon \leq p-1}\varepsilon ^{\frac{\theta}{p-\varepsilon}}\left\{\int_{\Omega}\mid\sum\limits_{ i\in I}
\mid\lambda_i\mid\chi_{x_i+U}\mid^{p-\varepsilon}dx\right\}^{\frac{1}{p-\varepsilon}}
&=\sup_{0<\varepsilon \leq p-1}\varepsilon ^{\frac{\theta}{p-\varepsilon}}\left\{\int_{\cup\left({x_i+U}\right)}\mid\sum\limits_{ i\in I}\mid\lambda_i\mid\chi_{x_i+U}\mid^{p-\varepsilon}dx\right\}^{\frac{1}{p-\varepsilon}}
\\
&=\sup_{0<\varepsilon \leq p-1}\varepsilon ^{\frac{\theta}{p-\varepsilon}}\left\{\sum\limits_{ i\in I}\int_{x_i+U}\mid\sum\limits_{ i\in I}\mid\lambda_i\mid\chi_{x_i+U}\mid^{p-\varepsilon}dx\right\}^{\frac{1}{p-\varepsilon}}
\\
&=\sup_{0<\varepsilon \leq p-1}\varepsilon ^{\frac{\theta}{p-\varepsilon}}\left\{\sum\limits_{ i\in I}\mid{\lambda_{i}\mid}^{p-\varepsilon}\int_{x_i+U}\chi_{x_i+U}dx\right\}^{\frac{1}{p-\varepsilon}}
\\
&=\sup_{0<\varepsilon \leq p-1}\varepsilon ^{\frac{\theta}{p-\varepsilon}}\mu\left(U\right)^{\frac{1}{p-\varepsilon}}\left\{\sum\limits_{ i\in I}\mid{\lambda_{i}\mid}^{p-\varepsilon}\right\}^{\frac{1}{p-\varepsilon}}
\\
&=\mu\left( U\right)\sup_{0<\varepsilon \leq p-1}\varepsilon ^{\frac{\theta}{p-\varepsilon}}\left\{\sum\limits_{ i\in I}\mid{\lambda_{i}\mid}^{p-\varepsilon}\right\}^{\frac{1}{p-\varepsilon}}
\\
&=\mu\left( U\right)\|\left\{\lambda_{i}\right\}_{i\in I}\|_{\ell^{p),\theta}}. \tag{9}
\end{align*}
The righ hand side of this equality is finite. This implies
\begin{align*}
 \left(L^{p),\theta}\right)_{d}\subset\ell^{p),\theta}\left(\ I\right).\tag{10}
\end{align*}
Conversely let $\Lambda=\left(\lambda_i\right)_{i\in I}\in\l^{p),\theta}\left(\ I\right).$ Again from (9) we have
\begin{align*}
 \left(L^{p),\theta}\right)_{d}\supset\ell^{p),\theta}\left(\ I\right).\tag{11}
\end{align*}
 Finally from (10) and (11)  we have
\begin{align*}
 \left(L^{p),\theta}\right)_{d}=\ell^{p),\theta}\left(\ I\right).\tag{12}
\end{align*}
\end{proof}
One can easily prove the following theorem if uses the norm of Grand 
amalgam apace instead of the norm of classical Wiener amalgam space and the
technics in ( $\left[ 8\right] ,$ Theorem 2 and $\left[ 15\right] ,$ Theorem
11.6.2)$.$

\begin{theorem}
If $\ \Psi =\left( \psi _{i}\right) _{i\in I}$ is a bounded uniform
partition unity and $V$ is a compact set containing $U,$ then 
\begin{equation*}
\left\Vert f\right\Vert _{W\left( L^{p),\theta },L^{p),\theta }\right)
}\asymp \left\Vert \sum_{i\in I}\left\Vert f\Psi _{i}\right\Vert _{p),\theta
}\chi _{V+y_{i}}\right\Vert _{q),\theta }.
\end{equation*}
\end{theorem}
\begin{theorem}
Let $1\leq p\leq \infty $ and let $\ \Psi =\left( \psi _{i}\right) _{i\in I}$
be a bounded uniform partition unity. Then 
\begin{equation*}
\left\Vert f\right\Vert _{W\left( L^{p),\theta },L^{q),\theta }\right)
}\asymp \left\Vert \sum\limits_{i\in I}\left\Vert f\psi _{i}\right\Vert
_{p),\theta }\chi _{U+y_{i}}\right\Vert _{q),\theta }=\left\Vert \left\{
\left\Vert f\psi _{i}\right\Vert _{p),\theta }\right\} _{i\in I}\right\Vert
_{\ell ^{q),\theta }}.
\end{equation*}
\end{theorem}

\begin{proof}
From Theorem 1 and Theorem 2, we have
\begin{align}
\left\Vert f\right\Vert _{W\left( L^{p),\theta },L^{q),\theta }\right)
}  &\asymp \left\Vert \sum\limits_{i\in I}\left\Vert f\psi _{i}\right\Vert
_{p),\theta }\chi _{U+y_{i}}\right\Vert _{q),\theta }  \tag{13}
\\
&=\sup_{0<\eta \leq q-1}\eta ^{\frac{\theta }{q-\eta }%
}\left\Vert \sum\limits_{i\in I}\left\Vert f\psi _{i}\right\Vert
_{p),\theta }\chi _{U+y_{i}}\right\Vert _{q-\eta }  \nonumber 
\\
&=\sup_{0<\eta \leq q-1}\eta ^{\frac{\theta }{q-\eta }}(|\int_{\Omega}\sum\limits_{i\in I} \|f\psi _{i}\|_{p),\theta }\chi _{U+y_{i}}| ^{q-\eta}dx) ^{\frac{1}{q-\eta }}\nonumber 
\\
&=\sup_{0<\eta \leq q-1}\eta ^{\frac{\theta }{q-\eta }%
}\left(\sum\limits_{i\in I}\left\Vert f\psi _{i}\right\Vert _{p),\theta
}^{q-\eta }\int_{\Omega}\left\vert \chi _{U+y_{i}}\right\vert ^{q-\eta}dx\right) ^{\frac{1}{q-\eta }}  \nonumber 
\end{align}
In particular,
\begin{equation}
\int_{\Omega }\left\vert \chi _{U+y_{i}}\right\vert ^{q-\eta }dx=\mu
\left( U+y_{i}\right) =\mu \left( U\right) .  \tag {14}
\end{equation}%
Then from $\left( 13\right) $ and $\left( 14\right) ,$%
\begin{eqnarray*}
\left\Vert \sum\limits_{i\in I}\left\Vert f\psi _{i}\right\Vert _{p),\theta
}\chi _{U+y_{i}}\right\Vert _{q),\theta }
&=&\sup_{0<\eta \leq q-1}\eta ^{\frac{\theta }{q-\eta }%
}\left( \sum\limits_{i\in I}\left\Vert f\psi _{i}\right\Vert _{p),\theta
}^{q-\eta }\mu \left( U\right) \right) ^{\frac{1}{q-\eta }} \\
&=&\mu \left( U\right) \sup_{0<\eta \leq q-1}\eta ^{\frac{%
\theta }{q-\eta }}\left( \sum\limits_{i\in I}\left\Vert f\psi
_{i}\right\Vert _{p),\theta }^{q-\eta }\right) ^{\frac{1}{%
q-\eta }} \\
&=&\mu \left( U\right) \left\Vert \left\{ \left\Vert f\psi _{i}\right\Vert
_{p),\theta }\right\} _{i\in I}\right\Vert _{\ell ^{q),\theta }}=\mu \left(
U\right) \left\Vert f\right\Vert _{W\left( L^{p),\theta },\ell ^{q),\theta
}\right) }.
\end{eqnarray*}%
This completes the proof.
\end{proof}
\section{Which  grand amalgam spaces are convolution algebras}
Let $G$ be locally compact Abelian group and $\mu $ its Haar measure.
Zelazko $\left[ 18\right] $ has proved that $L^{p}\left( G\right) $ is an
algebra if and only if $G$ is compact. Steward and Salem $\left[ 17\right] $
have determined necessary and sufficient conditions on $G$ for the Wiener
amalgam space $W\left( L^{p},L^{q}\right) $ to be an algebra under
convolution. In this chapter we have given similar results for
 the generalized grand Lebesgue space $L^{p),\theta }$ and the grand
 amalgam spaces.

\begin{theorem}
Let $G$ be a locally compact Abelian group, and $\mu$ its Haar measure. Then the generalized grand Lebesgue space $L^{p),\theta }$, $1<p<\infty $ is a Banach algebra under convolution if and only if $G$ is compact.
\end{theorem}
\begin{proof}
Let $G$ be a compact Abelian group. For the proof of  the generalized grand Lebesgue space $L^{p),\theta }$, $1<p<\infty $ is a Banach algebra under convolution, it is enough to show that 
\begin{equation*}
\left\Vert f\ast g\right\Vert _{p),\theta }\leq \left\Vert f\right\Vert
_{p),\theta }\left\Vert g\right\Vert _{p),\theta }
\end{equation*}%
for all $f,g\in L^{p),\theta }\left( G\right) .$ Since $G$ is compact, then $%
L^{p}\left( G\right) $ is a Banach convolution algebra for all $p,$ $1\leq p\leq\infty,$ by Zelasko $\left[ 18
\right] .$Thus $\left( L^{p-\varepsilon }\left( G\right) ,\left\Vert
.\right\Vert _{p-\varepsilon }\right) $ is a Banach algebra under
convolution for all $0<\varepsilon \leq p-1.$ Then for all $f,g\in
L^{p-\varepsilon }\left( G\right) ,$ we have 
\begin{equation*}
\left\Vert f\ast g\right\Vert _{p-\varepsilon }\leq \left\Vert f\right\Vert
_{p-\varepsilon }\left\Vert g\right\Vert _{p-\varepsilon }.
\end{equation*}%
Thus 
\begin{eqnarray*}
\left\Vert f\ast g\right\Vert _{p),\theta } &=&\sup_{0<\varepsilon \leq
p-1}\varepsilon ^{\frac{\theta }{p-\varepsilon }}\left\Vert f\ast
g\right\Vert _{p-\varepsilon }\leq \sup_{0<\varepsilon \leq p-1}\varepsilon
^{\frac{\theta }{p-\varepsilon }}\left\Vert f\right\Vert _{p-\varepsilon
}.\sup_{0<\varepsilon \leq p-1}\varepsilon ^{\frac{\theta }{p-\varepsilon }%
}\left\Vert g\right\Vert _{p-\varepsilon } \\
&\leq &\left\Vert f\right\Vert _{p),\theta }\left\Vert g\right\Vert
_{p),\theta }.
\end{eqnarray*}

Conversely Assume that $L^{p),\theta }( G) $ is a Banach algebra under convolution
for all $\theta \geq 0$ and, $1<p<\infty$. Since $L^{p),\theta }( G) $ reduces to $
L^{p}( G) $ when $\theta =0$, then $L^p(G)$ is a
Banach algebra under convolution. Thus $G$ is compact again from Zelasko $
\left[ 18\right] .$
\end{proof}

\begin{theorem}
The grand Wiener amalgam space $W\left( L^{p),\theta },L^{q),\theta }\right)
\left( G\right) ,p\geq 1,q>1$ is a Banach algebra \textit{under convolution
if and only if }$G$\textit{\ is compact.}
\end{theorem}
\begin{proof}
 It is known by $\left[ 17\right] $ that  $W\left(L^{p},L^{q}\right) \left( G\right) $ is a
 Banach algebra under convolution if and only if $G$ is compact. Let $G$ be a compact set. For the proof it is enough to show that
\begin{align*}
 \|f*g\| _{ W\left( L^{p),\theta} ,L^{q),\theta }\right) }\le \left\| f\right\|
_{ W\left( L^{p),\theta },L^{q),\theta }\right) }\|g\| _{ W\left( L^{p),\theta },L^{q),\theta }\right) }
\end{align*}
for all $ f,g \in W\left(L^{p),\theta },L^{q),\theta }\right) \left( G\right) .$
Since $G$ is compact,$W\left(L^{p-\varepsilon},L^{q-\eta}\right)$=$L^{p-\varepsilon}$=$L^{q-\eta}$ and the norms are equvalent.  Let $f,g \in W\left(L^{p),\theta },L^{q),\theta }\right) \left( G\right)$. Then by using Theorem 3, a simple calculation shows that
\begin{align*}
\| f * g\|_{W(L^{p),\theta}, L^{q),\theta})} 
&\leq\sup_{0<\eta\leq q-1}\eta^{\frac{\theta}{q-\eta}}\sup_{0<\varepsilon \leq p-1}\varepsilon ^{\frac{\theta }{p-\varepsilon }}\|f\|_{p-\varepsilon}\|g\|_{q-\eta}
\\
&=\| f\|_{W(L^{p),\theta}, L^{q),\theta})}\| g\|_{W(L^{p),\theta}, L^{q),\theta})}
\end{align*} 
\quad Conversely assume that  $ W\left(L^{p),\theta },L^{q),\theta }\right) \left( G\right) $ is a Banach convolution algebra for all $\theta\geq0.$ Then it is Banach convolution algebra for $\theta=0.$ But if $\theta=0,$ then the grand amalgam space  $ W\left(L^{p),\theta },L^{q),\theta }\right) \left( G\right) $ reduces to classical Wiener amalgam space  $W\left(L^{p},L^{q}\right) \left( G\right).$ So, $W\left(L^{p},L^{q}\right) \left( G\right) $ becomes a Banach convolution algebra. Then by Theorem 1, in $ \left[17\right],$ $G$ becomes compact.

\end{proof}


\begin{thebibliography}{99}
\bibitem{Ana1} Anatriello G, Chil R and Fiorenza A. Identification of fully
Measurable Grand Lebesgue Spaces. Journal of Function Spaces, Vol. 2017.

\bibitem{Ana2} Anatriello G. Iterated grand and small Lebesgue spaces. Collect.
Math. 2014; 64: 273--284.

\bibitem{} Capone C, Formica MR, Giova R. Grand Lebesgue spaces with respect
to measurable functions. Nonlinear Analysis 2013; 85:  125--131.

\bibitem{} Capone C, and Fiorenza A. On small Lebesgue spaces. Journal of
function spaces and applications  2005; 3 (1t): 73--89.

\bibitem{} Castillo R. E, Raferio H. Inequalities with conjugate exponents
in grand Lebesgue spaces. Hacettepe Journal of Mathematics and Statistics 
2015; 44 (1) :  33-39.

\bibitem{} Castillo R. E, Raferio H. An Introductory Course in Lebesgue
Spaces. Springer International Publishing Switzerland  2016.

\bibitem{DK} Danelia N.,  Kokilashvili V. On the approximation of periodic
functions within the frame of grand Lebesgue spaces. Bulletin of the
Georgian national academy of sciences  2012; 6(2): 11--16.

\bibitem{Fei1} Feichtinger HG, Banach convolution algebras of Wiener's type,
Proc. Conf. " Functions, Series, Operators", Budapest,  1980, Colloquia
Math. Soc. J. Bolyai, North Holland Publ. Co., Amsterdam- Oxford- New York 
1983: 509--524.

\bibitem{Fei2} Feichtinger HG and Gr\"{o}bner P. Banach Spaces of Distributions
Defined by Decomposition Methods I. Math.Nachr.;  1985;  123:  97--120.

\bibitem{Fio1} Fiorenza A, and Karadzhov GE. Grand and small Lebesgue spaces and
their analogs, Journal for Analysis and its Applications 2004; 23 (4) :  657--681.

\bibitem{Fio2} Fiorenza A. Duality and reflexity in grand Lebesgue spaces,
Collect. Math. 2000; 51  ( 2) :  131--148.

\bibitem{} Greco L, Iwaniec T, Sbordone C. Inverting the p-harmonic
operator, Manuscripta Math.1997;92:259--272.

\bibitem{Gur1} Gurkanli AT. Inclusions and the approximate identities of the
 generalized grand Lebesgue spaces, Turk J Math.2018 ;42:3195--3203.

\bibitem{Gur2} Gurkanli AT. On the grand Wiener amalgam spaces, arXiv: 1811.08629v1  [math.FA]21 Nov 2018, (Submitted to a Journal)

\bibitem{} Heil C, An Introduction to Weighted Wiener Amalgams, In: Wavelets
and their Applications (Chennai, 2002), Allied Publishers, New Delhi, 2003:
183-216.

\bibitem{} Iwaniec T, Sbordone C. On the integrability of the Jacobian under
minimal hypotheses, Arc. Rational Mech. Anal. 1992;119:129--143.

\bibitem{} Stewart J, and Watson S. Which amalgams are convolution algebras,
Proceedings of the American Mathematical Society, Volume 93, Number 
4,1958,621-627.

\bibitem{} Zelazko W. On the algebras $L^{p}$ of locally compact groups,
Colloquium Mathematicum, Vol. VIII, Fasc. 1,1961,115-120.
\end{thebibliography}
\end{document}